\documentclass[pdflatex,sn-mathphys-num]{sn-jnl}
\usepackage{graphicx}%
\usepackage{multirow}%
\usepackage{amsmath,amssymb,amsfonts}%
\usepackage{amsthm}%
\usepackage{mathrsfs}%
\usepackage[title]{appendix}%
\usepackage{xcolor}%
\usepackage{textcomp}%
\usepackage{manyfoot}%
\usepackage{booktabs}%
\usepackage{algorithm}%
\usepackage{algorithmicx}%
\usepackage{algpseudocode}%
\usepackage{listings}%
\usepackage{letltxmacro}
\makeatletter
\let\oldr@@t\r@@t
\def\r@@t#1#2{%
\setbox0=\hbox{$\oldr@@t#1{#2\,}$}\dimen0=\ht0
\advance\dimen0-0.2\ht0
\setbox2=\hbox{\vrule height\ht0 depth -\dimen0}%
{\box0\lower0.4pt\box2}}
\LetLtxMacro{\oldsqrt}{\sqrt}
\renewcommand*{\sqrt}[2][\ ]{\oldsqrt[#1]{#2}}
\makeatother
\usepackage{enumitem}

\newtheorem{thm}{Theorem}%
\newtheorem*{thms}{Theorem}%
 \newtheorem{lem}{Lemma}
 \newtheorem{prop}{Proposition}
\newtheorem{cor}{Corollary}
\newtheorem{asm}{Assumption}

\theoremstyle{thmstyletwo}%
\newtheorem{rem}{Remark}%
\newtheorem*{rem*}{Remark}%

\theoremstyle{thmstylethree}%

\newcommand{\dist}{{\rm dist}}

\newcommand{\Real}{\mathbb R}
\newcommand{\Natural}{\mathbb N}

\newcommand\Complex{\mathbb C}
\newcommand{\bbP}{\mathbb P}
\newcommand{\bbQ}{\mathbb Q}

\newcommand{\I}{i}

\newcommand{\normf}[1]{\|#1\|_\bbf}

\newcommand{\ttp}{\tilde{p}}
\newcommand{\ttq}{\tilde{q}}
\newcommand{\tth}{\tilde{h}}

\newcommand{\bbg}{\mathfrak{g}}
\newcommand{\bbf}{\mathfrak{f}}
\newcommand{\bbh}{\mathfrak{h}}
\newcommand{\hbbh}{\hat{\bbh}}

\newcommand{\N}{{\cal N}}

\newcommand{\calm}{{\cal M}}

\newcommand{\calg}{{\cal G}}
\newcommand{\calh}{{\cal H}}

\newcommand{\cT}{\mathcal{T}}
\newcommand{\cS}{\mathcal{S}}
\newcommand{\cL}{\mathcal{L}}

\newcommand{\cH}{\mathcal{H}}

\newcommand{\cM}{\mathcal{M}}

\newcommand{\cF}{\mathcal{F}}

\newcommand\Normal{\mathcal{N}}

\newcommand{\ex}{{\mathbb E}}
\newcommand{\dx}{{\mathbb D}^2}

\newcommand{\Prob}{{\mathbf{P}}}

\newcommand{\al}{\alpha}                

\newcommand{\const}{{\rm const\,}}

\newcommand{\Det}{{\rm Det\,}}

\newcommand{\commentOUT}[1]{}

\begin{document}

\title[GHMC]{On Hamiltonian Monte Carlo for Gaussian Random Variables with Random Hamiltonians}

\author{\fnm{Yingdong} \sur{Lu}\email{yingdong@us.ibm.com}}

\author{\fnm{Tomasz J.} \sur{Nowicki}\email{tnowicki@us.ibm.com}}

\affil{{\center \orgname{IBM} \orgdiv{T.J. Watson Research Center}\\ \orgaddress{\street{1101 Kitchawan Rd}, \city{ Yorktown Heights}, \postcode{NY 10598}, \country{U. S. A.}}}}


\abstract{We study a family of (multivariate-)Gaussian Hamiltonian Monte Carlo (GHMC) operators and prove that the family of Gaussian distributions and their mixtures are invariant under such operators. Furthermore, each such operator is a contraction on the space of parameters and an explicit formulae are derived.  These results then enable us to analyze the dynamics and convergences of independent and identically distributed random sequences of such operators.
}




\maketitle

\section{Introduction}
\label{sec:intro}

Hamiltonian Monte Carlo (HMC) algorithm aims at effectively sampling or estimating probability distributions known up to their partition function, in other words, their normalizing constants. This is a very common problem in statistical mechanics and Bayesian analysis. A critical component of the algorithm is the calculation of the gradient of the logarithm of the density function, which can be obtained without knowledge of the partition function itself. Unfortunately, this calculation is also often the most computationally expensive component. As noted  in~\cite{chen2014stochastic}, for contemporary 
applications of HMC, given the magnitude of data to be processed, exact gradient calculations become infeasible, and a random variable whose statistics can reflect the true gradient is used in its place instead. This requires understanding of the HMC algorithm whose potential energy function (exponent of the \emph{target} density function with a changed sign) is randomly generated. Note that if the target distribution is given (deterministic), it is known that even when the kinetic energy function (exponent of the \emph{auxiliary} densities with a changed sign) may change, the HMC algorithm still converges to the distribution of the normalized target~\cite{GHOSH2025134952}. However, the randomness of the potential energy function requires a different treatment, which is the main purpose of this paper. 

Most of the results in this paper are obtained under the assumption that the target distributions are multivariate Gaussian. While in this case the limit distribution is not difficult to know fully,  the explicit characterization of the steps of HMC obtained in Theorem~\ref{thm: GHMC on cG} is not available in the literature. We followed the \emph{operator formulation} of HMC in~\cite{GHOSH2022107811} which facilitates the derivation of the results. Equipped with these explicit characterizations, we are able to demonstrate in Theorem~\ref{thm:convergence_SHMC} that the algorithm, without any adjustment (such as Metropolis-Hastings steps) produces a sequence of random variables that converge in distribution, and explicitly characterize the limit distribution with properly chosen step size in the univariate case in Theorem~\ref{thm:Xinfty}. In addition, we demonstrate the boundedness of the pointwise asymptotic behavior of the outcome random variables following Proposition~\ref{prop:dist to convex}. Previous convergence results on such algorithms are limited to either the case required Metropolis-Hastings adjustment, see, e.g.~\cite{ZouGu2021} or appealing to its continuous counterpart, see. e.g.~\cite{gao21sghmc}. Furthermore, this characterization allows us to construct a metric on the space of normal distributions, such that HMC dynamics will always follow geodesics. Then,  one can utilize known relationships between geodesics in the space of probability measures and optimal transport in order to provide another view of HMC, as discretized optimal transport.

\medskip
The rest of the paper is organized as following: in Sec.~\ref{sec:GHMC}, we establish in Theorem~\ref{thm: GHMC on cG} the main result the HMC operator with Gaussian target and auxiliary distributions keeps the set of Gaussian distributions invariant and provide the formulae for the change of their moments under HMC. In Proposition~\ref{prop: contraction} we conclude that HMC is a contraction in the space of  their moments and in Proposition~\ref{prop:dist to convex} we characterize the set of possible accumulation points for random iterations, that is when targets are chosen randomly. In Sec.~\ref{sec:SHMC}, we establish convergence results for HMC with random Gaussian targets.

\section{Hamiltonian Monte Carlo for Gaussian Target and Auxiliary} 
\label{sec:GHMC}

First, we will recall some fundamental results on functional representation of HMC; then, we derive the main theorem for the section; finally, we discuss consequences of the theorem. 

\subsection{The general case in $\Real^d$}

We recall the basic results of~\cite{GHOSH2022107811} on functional formulation of HMC in a concrete setting of $\Real^d$.

Suppose that $\bbf$ is a \emph{target} function $0< \bbf:\Real^d\to\Real$ with $0<\int_{\Real^d}\bbf(q)\,dq<\infty$. 
Next, let $\bbg$ be the \emph{auxiliary} density $0\le \bbg:\Real^d\to\Real$ with $\int_{\Real^d}\bbg(p)\,dp=1$.
Let's suppose that there exists a measurable, invertible \emph{motion} $H:\Real^{2d}\to \Real^{2d}$, $(Q,P)=H(q,p)$,  which satisfies two invariance conditions and an irreducibility condition: for all $q,p$ and measurable function $\phi$
\begin{align}\label{eqn:inv-cond}
    \bbf(Q)\cdot \bbg(p)&=\bbf(Q)\cdot \bbg(P)&\iint_{\Real^{2d}} \phi(Q,P)\, dqdp =&\iint_{\Real^{2d}}\phi(q,p)\, dqdp \\
    Q(q, \bbP)&=\bbQ    \nonumber
\end{align}
    then in  the space $\cL^2_\bbf=\{\bbh: ||\bbh||^2_\bbf=\int_{\Real^{d}}\frac{\bbh^2(q)}{\bbf(q)}\, dq<\infty\}$
 the linear operator 
 \begin{align*}
   \cT:&\cL^2_\bbf\to\cL^2_\bbf, \qquad \cT(\bbh)(q):=\int_{\Real^{d}} \bbh(Q)\bbg(P)\, dp   
 \end{align*}
  is well defined with $\bbf$ is its unique fixed point up to a constant. The operator $\cT$ is bounded: 
  $\normf{\cT\bbh}\le\normf{\bbh}$, with equality only when $\bbh=\const \bbf$. Its adjoint operator $\cT^*$ is defined by the inverse motion $H^{-1}$. If $\cT=\cT^*$ then for each $\bbh$ iterations $\cT^n(\bbh)$ converge strongly to $\bbf\cdot \frac{\int\bbh}{\int \bbf}$ and $\bbf$ defines the unique eigen-direction with eigenvalue 1. Even in the case when $\cT$ is not self-adjoint the analogous convergence result holds for the self-adjoint operator $\mathcal{S}=\cT\circ\cT^*$.
\begin{rem*}
  The invariance conditions are satisfied when the motion is given by the Hamiltonian equations with the Hamiltonian energy $\cH(Q,P)=-\log(\bbf(Q)\bbg(P)$ and 
  
  \begin{align}\label{eqn: Ham Motion}
  \frac{\partial Q}{\partial t}&=\phantom{-}\frac{\partial \cH(Q,P)}{\partial P}&
  \frac{\partial P}{\partial t}&=-\frac{\partial \cH(Q,P)}{\partial Q}\\
  H(q,p)&=(Q(t), P(t))& Q(0)&=q,\quad P(0)=p\,,    
  \end{align}
  for some time $t$. In this case, the inverse motion is obtained by the opposite time. A~sufficient condition for $\cT$ to be self-adjoint is a (possibly translated) evenness of $\bbg$.
\end{rem*}

\subsection{Hamiltonian Monte Carlo, Gaussian edition}\label{subsec:GHMC}
 The \emph{Gaussian HMC}, or GHMC for short, is an HMC operator $\cT$, where both the target  density $\calg(d)\ni\bbf\sim\Normal(\mu_\bbf, \Sigma_\bbf)$ and the auxiliary density $\calg(d)\ni\bbg\sim\Normal(\mu_\bbg, \Sigma_\bbg)$ are $d$-multivariate Gaussian and their \emph{covariance matrices commute}, $\Sigma_\bbf\cdot\Sigma_\bbg=\Sigma_\bbg\cdot\Sigma_\bbf$.
 
     \begin{rem}[Properties of covariances]\label{rem:properties of cov}
     \ 
     \begin{enumerate}
     \item
         For  $\bbh\in \calg(d)$ its (symmetric) covariance matrix  is invertible, we explicitly exclude the matrices with non-positive eigenvalues. 
    \item         
         For a symmetric positive definite matrix $M$ the functions 
         $M^{-1}$, $\sqrt{M}$, $\cos(M)$, and $\sin(M)$ are uniquely defined, symmetric and positive definite matrices, see, e.g.~\cite{higham2008functions}. 
    \item     
         When the matrices $\Sigma_\bbf, \Sigma_\bbg$ commute, so do all their analytic functions, in particular the ones mentioned above.
    \item The commutativity condition is satisfied in a standard situation where $\Sigma_\bbg$ is the identity matrix times a scalar.    
     \end{enumerate}
     \end{rem}
        We recall that for $\calg(d)\ni\bbh\sim\Normal(\mu_\bbh, \Sigma_\bbh)$,  the normalizing constant is $N_\bbh=(2\pi)^{-d/2}\Det(\Sigma_\bbh)^{-1/2}$ and 
        $\bbh(q)=N_\bbh\exp\left[-\frac{1}{2}(q-\mu_\bbh)^\top \Sigma_\bbh^{-1}(q-\mu_\bbh)\right]$.
   For ease of exposition, we shall use the following notation:
    \begin{align}\nonumber
        F&=\Sigma_\bbf^{-1} & G&=\Sigma_\bbg^{-1} & H&=\Sigma_\bbh^{-1}\\
        A&=\sqrt{F^{-1}G}& C&=\cos(FAt)&S&=\sin(FAt)
        \label{eqn:notation}
        \\
       \ttq&=q-\mu_\bbf& \ttp&=p-\mu_\bbg & \tth&=\mu_\bbh-\mu_\bbf\,.
       \nonumber
    \end{align}
    Remark~\ref{rem:properties of cov} asserts that the symmetric positive definite matrices  $A, C, S$, $t>0$, are uniquely defined and commute with each other and with the matrices  $F$ and $G$. Some rudimentary identities that will be frequently applied include: $FA=\sqrt{FG}=\sqrt{GF}=AF$, 
     $G=A^2F$ and $S^2+C^2=I$, where $I$ is the identity matrix.
  
\begin{prop}[Gaussian motion]\label{prop:Gaussian Motion}
    For GHMC with Hamiltonian function $\calh(Q,P)=\frac{1}{2}\left((Q-\mu_\bbf)^\top F (Q-\mu_\bbf)+(P-\mu_\bbg)^\top G (P-\mu_\bbg)\right)+\const$, the Hamiltonian motion with time parameter $t$ is given~by:
    \begin{align*}
      Q(t)&=\mu_\bbf+\cos(AF t) (q-\mu_\bbf)+A\sin(AF t) (p-\mu_\bbg)
         =\mu_\bbf+C \ttq+AS\ttp\\
      P(t)&=\mu_\bbg-A^{-1}\sin(AF t) (q-\mu_\bbf)+\cos(AF t) (p-\mu_\bbg)
      =\mu_\bbg-A^{-1}S\ttq+C\ttp\,,
        \end{align*}
        where we used the notation~\eqref{eqn:notation}.
    \end{prop}
    
    \begin{proof}
        The derivatives of the trigonometric functions of the matrices are in the form of $\partial/\partial t \sin(Mt)=M\cos(Mt)$ and $\partial/\partial t \cos(Mt)=-M\sin(Mt)$, see e.g.~\cite{higham2008functions}. Thus, 
        \begin{align*}
   \frac{\partial}{\partial t}{Q}&=\phantom{-}\frac{\partial}{\partial P}{\calh(Q,P)}=\phantom{-}\Sigma_\bbg^{-1}(P-\mu_\bbg)
   =\phantom{-}A^2F(P-\mu_\bbg)
   \\
   \frac{\partial}{\partial t}{P}&=-\frac{\partial}{\partial Q}{\calh(Q,P)}=-\Sigma_\bbf^{-1}(Q-\mu_\bbf)
   =-F(Q-\mu_\bbf)   \,.
        \end{align*}
    The expressions of $Q(t)$ and $P(t)$ follow directly from standard calculations based on the above Hamiltonian equations. 
    \end{proof}

\noindent
This leads immediately to,   
   \begin{cor}[GHMC acting on $\calg(d)$]\label{cor:GHMC on cD}
    For $\calg(d)\ni\bbh\sim\Normal(\mu_\bbh, \Sigma_\bbh)$ the GHMC transformation $\cT$ with the target $\bbf$ and  the auxiliary $\bbg$ is given by:
   \begin{align*}
       \hbbh(q)=\cT\bbh(q)&=\int_{\Real^d}\left(N_\bbh
       \exp\left[-\frac{1}{2}(Q-\mu_\bbh)^\top(\Sigma_\bbh)^{-1}(Q-\mu_\bbh)\right]
       \right.\times
       \\
       &\phantom{=\int_{\Real^d}(}
       \times\left.
       N_\bbg\exp\left[-\frac{1}{2}(P-\mu_\bbg)^\top(\Sigma_\bbg)^{-1}(P-\mu_\bbg)\right]\right)dp
       \\
       &=N_\bbh N_\bbg\int_{\Real^d}\left(
       \exp\left[-\frac{1}{2}(C\ttq+A S\ttp-\tth)^{\top}H(C\ttq+A S\ttp-\tth)\right]
       \right.\times
       \\
       &\phantom{=N_\bbh N_\bbg\int_{\Real^d}(}
       \times\left.
       \exp\left[-\frac{1}{2}(-A^{-1}S\ttq+C\ttp)^{\top}G(-A^{-1}S\ttq+C\ttp)\right]
       \right)dp \,,
   \end{align*}
   with normalization factors
   $N_\bbh\cdot N_\bbg=(2\pi)^{d/2}(\Det\Sigma_\bbh)^{-1/2}\cdot(2\pi)^{d/2}(\Det  \Sigma_\bbg)^{-1/2}$.  Again we used the notation~\eqref{eqn:notation}.    
   \end{cor}

\subsection{Invariance and contraction, one step}
\label{sec:invariance}

 Our main result of this section is the following:
    \begin{thm}\label{thm: GHMC on cG}
    The family $\calg(d)$ of multivariate normal distribution on $\Real^d$ is invariant under the GHMC transformation $\cT$ with target $\bbf\sim\Normal(\mu_\bbf, \Sigma_\bbf)$ and auxiliary $\bbg\sim\Normal(\mu_\bbg, \Sigma_\bbg)$ with commuting covariance matrices. For any 
     $\calg(d)\ni\bbh\sim \Normal (\mu_\bbh, \Sigma_\bbh)$ its image $\cT\bbh=\hbbh\sim\Normal(\mu_{\hbbh}, \Sigma_{\hbbh})$ satisfies:    
    \begin{align}\label{eqn:thm:mv contraction}
    \mu_{\hbbh}-\mu_\bbf&=C(\mu_\bbh-\mu_\bbf)&\text{and}&&\Sigma_{\hbbh}-\Sigma_{\bbf}&=C^\top(\Sigma_{{\bbh}}-\Sigma_{\bbf})C
    \\ \nonumber
    \text{or}\qquad\mu_{\hbbh}&=(I-C)\mu_\bbg+C\mu_\bbh&\text{and}&&
    \Sigma_{\hbbh}&=C^\top\Sigma_{{\bbh}}C+S^\top\Sigma_{\bbf}S 
    \,,
    \end{align}
    where the matrices $C=\cos( D)$ and $S=\sin(D)$  with $D= t\sqrt{\Sigma_\bbf^{-1}\cdot\Sigma_\bbg^{-1}}$ for some $t>0$ satisfy $0<C<I$ and $C^2+S^2=I$.
    \end{thm}
    The proof of Theorem~\ref{thm: GHMC on cG} consists mainly of lengthy 
    elementary 
    calculations and terminates in Corollary~\ref{prop: contraction} below. 
 \commentOUT{Seems superfluous as it is defined inside Lemma 1 \\ 
Given positive defined symmetric matrices $F,A, C,S$ and $H$ and a vector $\tilde{h}$ define 
\begin{itemize}
\item
positive defined symmetric matrices:  $K:=A(SHS+CFC)A$ and $Y:= \left(CH^{-1}C+SF^{-1}S\right)^{-1}$;
\item
a matrix: $X:=K^{-1} SA(H-F)C$;
\item
vectors: $x:=K^{-1}SA H \tilde{h}$, $y:=C^\top \tilde{h}$;
\item
a scalar: $\zeta:=h^\top H h-x^\top K x- y^\top Y y$.
\end{itemize}
}
First, we have,
\begin{lem}[Representation of a sum of quadratic terms]\label{lem:repr}
     Suppose that the symmetric, positive defined matrices $F, A, S, C$ are pairwise commuting and  satisfy $S^\top S+C^\top C=I$,
     then for any symmetric positive defined matrix $H$ (not necessarily commuting with the other matrices) 
     we have for any $(\ttq, \ttp)$:
     
\begin{align}
\nonumber
    &(C\ttq -AS\ttp -\tth)^\top H (C\ttq -AS\ttp -\tth)+(S\ttq +AC\ttp )^\top F(S\ttq +AC\ttp ) \nobreak
    \\ 
    &\equiv(\ttp  +X\ttq  - x)^\top K (\ttp + X\ttq  - x)+
    (\ttq -y)^\top Y (\ttq -y) +\zeta\,, \label{eqn:quadratic_form}
 \end{align}
where:
\begin{align*}
     Y&= \left(CH^{-1}C+SF^{-1}S\right)^{-1} & y&=C\tth\\
     K&=A(SHS+CFC)A & x&=A^{-1}(SHS+CFC)^{-1}SH \tth\\
     X&=-A^{-1}(SHS+CFC)^{-1}S(F-H)C & \zeta&=0\,.
\end{align*}
\end{lem}
\begin{rem}
Lemma~\ref{lem:repr} is the key technical lemma whose aim is to group the dependence on $\ttp$ in one quadratic term, possibly including some dependence on $\ttq$,  and to make sure that the second quadratic term depends only on $\ttq$. This will facilitate the integration presented in Corollary~\ref{cor:GHMC on cD}. 
\end{rem}

\begin{proof}
The proof needs several calculations comparing the right and the left sides of the equality. We often invoke the assumption that $C,S,A,F$ are symmetric, invertible and commutative and the identity $CFC+SFS=(S^2+C^2)F=F$.

Naturally, the expression~\eqref{eqn:quadratic_form} is obtained through comparing the coefficients of the terms in the zeroth, first and second orders of the $\ttp$ and $\ttq$ terms. 

\noindent
{\bf Second order terms:}
Comparing the coefficients for second order term of $\ttp$, that is $$\ttp^\top K \ttp \equiv \ttp^\top S A H  AS \ttp +\ttp^\top C A F A C \ttp, $$ we get $K=A(S H S+C F C) A$.   

From the cross term of $\ttp$ and $\ttq$. that is 
$$\ttp^\top KX \ttq \equiv-\ttp^\top S A HC\ttq +\ttp^\top C A F S \ttq, $$
we get 
\begin{align*}
KX= & CAFS-AS HC =  AS(F- H)C.
\end{align*}
Therefore,
\begin{align*}
X=&K^{-1}KX=A^{-1}(SHS+CFC)^{-1}S(F-H)C, \\
X^\top KX=&C(F-H)\cdot(SHS+CFC)^{-1}\cdot S(F-H)C.
\end{align*}
The third second order term gives us 
\begin{align*}
\ttq^\top X^\top K X \ttq  +\ttq^\top Y\ttq  \equiv \ttq^\top C H C \ttq + \ttq^\top S F S \ttq
\end{align*}
Hence, we get
\begin{align*}       
Y&=CHC+ SFS - (C(H-F)S)\cdot(SHS+CFC)^{-1}\cdot(S(H-F)C.
\end{align*}
Therefore, the identity $CFC+SFS=(S^2+C^2)F=F$ gives us 
 \begin{align*}       
    Y-F&= CHC-CFC -(C(H-F)S)\cdot(SHS+CFC)^{-1}\cdot(S(H-F)C\\
    &=C(H-F)\left(I-S(SHS+CFC)^{-1}\cdot(S(H-F)\right)C\\
    &=C(H-F)S(SHS+CFC)^{-1}\cdot\left((SHS+CFC)S^{-1}-S(H-F)\right)C\\
    &=C(H-F)S(SHS+CFC)^{-1}\cdot\left(CFC+SFS\right)S^{-1}C\\
    &=C(H-F)S(SHS+CFC)^{-1}\cdot S^{-1}C F.
\end{align*}
Thus,
\begin{align*}
    Y&=(C(H-F)S(SHS+CFC)^{-1}\cdot S^{-1}C+I) F\\
    &\stackrel{(1)}{=}\left(C(H-F)S+ C^{-1}S(SHS+CFC)\right)\cdot (SHS+CFC)^{-1}\cdot S^{-1}C F\\
    &=\left(CHS-CFS+C^{-1}S^2HS+SFC\right)\cdot (SHS+CFC)^{-1}\cdot S^{-1}C F\\
    &=C^{-1}\left(C^2HS+S^2HS\right)\cdot (SHS+CFC)^{-1}\cdot C FS^{-1}\\
    &\stackrel{(2)}{=}C^{-1}HS\cdot (SHS+CFC)^{-1}\cdot S^{-1}FC\,,
    \end{align*}
where (1) is the result of writing $I= C^{-1}S(CFC)(CFC)^{-1}S^{-1}C$, and (2) follows from $C^2+S^2=I$.
By commutativity, we obtain, 
\begin{align*}
    Y^{-1}&=C^{-1}F^{-1}S(SHS+CFC) S^{-1}H^{-1}C\\
    &=C^{-1}F^{-1}S^2HSS^{-1}H^{-1}C+C^{-1}F^{-1}SCFCS^{-1}H^{-1}C\\
    &=SF^{-1}S+C H^{-1}C\,.
\end{align*}

\noindent
{\bf First order terms:}

Similarly, comparing the first order terms gives us, 
\begin{align*}
x=K^{-1}S A H \tth=-A^{-1}(SHS+CFC)^{-1}S H\tth,
\end{align*}
and
\begin{align*}
q^\top X^\top K(-x)+q^\top Y(-y)\equiv q^\top C^\top H (-\tth).
\end{align*}
Therefore, 
  \begin{align*}   
  Yy&=C H \tth-X^\top K x=C\left(I -C(H-F)S(SHS+CFC)^{-1}S\right)H \tth
  \\
  &\stackrel{(1)}{=}C\left(S^{-1}(SHS+CFC) -(H-F)S\right)(SHS+CFC)^{-1}SH \tth\\
  &=C\left(HS+S^{-1}CFC-HS+FS)\right)(SHS+CFC)^{-1}SH \tth\\
  &=CS^{-1}\left(CFC+SFS\right)(SHS+CFC)^{-1}SH\tth\\
  &=CS^{-1}F(SHS+CFC)^{-1}SH\tth\\
  &=C(F^{-1}S)^{-1}(SHS+CFC)^{-1}(H^{-1}S^{-1})^{-1}\tth\\
  &=C\left(H^{-1}S^{-1}SHSF^{-1}S +H^{-1}S^{-1}CFCF^{-1}S\right)^{-1}\tth
  \\
  &=C\left(SF^{-1}S+H^{-1}C^2\right)^{-1}\tth=
  C\left(C^{-1}(SF^{-1}S)C+C^{-1}(CH^{-1}C) C\right)^{-1}\tth\\
  &=C\left(C^{-1}(SF^{-1}S+CH^{-1}C)^{-1}C\right)\tth=Y C\tth.
\end{align*}
where we write $I=S^{-1}(SHS+CFC)(SHS+CFC)^{-1}S$ for (1) to hold. 
  Hence, we can conclude,
  \begin{align*}
  y&=Y^{-1} Y C\tth=C\tth\,.
  \end{align*}

\noindent
{\bf Zeroth order terms:}

Finally, comparing the zeroth order terms gives us, 
$x^\top K x+ y^\top Y y+\zeta \equiv \tth^\top H\tth$.
Therefore, \quad $\zeta=\tth^\top H\tth-x^\top K x- y^\top Y y$. Expanding it, we get
\begin{align*}
  \zeta&= \tth^\top\left( H -(H^\top A^\top S^\top K^{-\top})K 
    (K^{-1}SA H)-C(C^\top H^{-1} C+S^\top F^{-1}S )^{-1}C^\top)\right) \tth
    \\
    &=\tth^\top\left(H -HS (SHS+CFC)^{-1}SH-( H^{-1} +C^{-1}S F^{-1}SC^{-1} )^{-1}\right) \tth
    \\ 
    &=\tth^\top\left(
    H-H(H+S^{-1}CFCS^{-1})^{-1}H-H( I +C^{-1}SF^{-1}SC^{-1}H )^{-1}\right) \tth
    \\
    &=\tth^\top  H\left(I -\underbrace{(H+S^{-2}C^2F)^{-1}H}_{L}-\underbrace{(I+C^{-2}S^2F^{-1}H)^{-1}}_{R}
    \right)\tth
    \\
    &=\tth^\top H \cdot\left(\underbrace{(H+S^{-2}C^2F)^{-1}H}_{L}\right)\cdot Z \cdot\left(\underbrace{(I+C^{-2}S^2F^{-1}H)^{-1}}_{R}\right)\cdot\tth\,, \qquad\text{where}
    \\
    Z=&
    \underbrace{H^{-1}(H+S^{-2}C^2F)}_{L^{-1}}\cdot\underbrace{(I+C^{-2}S^2F^{-1}H)}_{R^{-1}}
    -\underbrace{(I+C^{-2}S^2F^{-1}H)}_{R^{-1}}
    -\underbrace{H^{-1}(H+S^{-2}C^2F)}_{L^{-1}}
    \\
    &=H^{-1}(H+S^{-2}C^2F)\cdot(C^{-2}S^2F^{-1}H)-(I+C^{-2}S^2F^{-1}H)\\
    &=(I+H^{-1}S^{-2}C^2F)\cdot(C^{-2}S^2F^{-1}H)-(I+C^{-2}S^2F^{-1}H)\\
    &=C^{-2}S^2F^{-1}H+H^{-1}\underbrace{S^{-2}C^2FC^{-2}S^2F^{-1}}_{\text{commute}}H-I-C^{-2}S^2F^{-1}H=\mathbf{0}
\end{align*}
Thus, we can conclude that $\zeta=\tth^\top \cdot\mathbf{0} \cdot \tth=0\,$. This concludes the proof of the lemma. 
\end{proof}

Using the notations from Lemma~\ref{lem:repr}, we have, 
 \begin{cor}[New Gaussian distribution]\label{cor:new distr}
    For each $\ttq$, the expression $\ttp\mapsto(\ttp  +X\ttq  - x)^\top K (\ttp + X\ttq  - x)$ can be treated as the negative exponent of a Gaussian density $\Normal(x-X\ttq, K^{-1})$ with the normalizing constant 
    $(2\pi)^{d/2}\sqrt{\Det K}$. The expression $\ttq\mapsto(\ttq  - y)^\top Y (\ttq-y)$ can be treated as the negative of the exponent of a Gaussian density $\Normal(y, Y^{-1})$ with the normalizing constant 
    $(2\pi)^{d/2}\sqrt{\Det Y}$. We have:
    \begin{align*}
        \Det (K)\cdot \Det (Y) =\Det(H)\cdot \Det(G)=  \Det \Sigma_\bbh^{-1}\cdot   \Det\Sigma_\bbg^{-1}\,.
    \end{align*}
\end{cor}
\begin{proof}
Recall that $Y=(SF^{-1}S+CH^{-1}C)^{-1}=C^{-1}HS(SHS+CFC)^{-1}S^{-1}FC$
so that $\Det(Y)=\Det(H)\cdot\Det(F)\cdot \Det(SHS+CFC)^{-1}$ and that $A^2=GF^{-1}$ so that $\Det A^2=\Det G\cdot \Det F^{-1}$:
    \begin{align*}
        \Det (K)\cdot \Det (Y)&= \left(\Det(A)\cdot \Det(SHS+CFC)\right)\cdot \\
        &\cdot 
        \left(\Det(H)\cdot\Det(F)\cdot \Det(SHS+CFC)^{-1}\right)
        \\
        &=\Det G\cdot \Det F^{-1}\cdot 1\cdot  \Det(H)\cdot\Det(F)=
        \Det(G)\cdot\Det(H)
    \end{align*}
\end{proof}

\begin{rem}
    The last result follows also from the fact that the integral is conserved under the HMC transformations, see~\cite{GhoshLuNowicki+2022}. This invariance gives another proof that the scalar $\zeta=0$ in the previous Lemma.
\end{rem}

This concludes the proof of Theorem~\ref{thm: GHMC on cG}.

 \begin{prop}[Contraction of the moments]\label{prop: contraction}
     Under multivariate GHMC for Gaussian input $\bbh\sim \Normal(\mu_\bbh, \Sigma_\bbh)$ its image $\hbbh=\cT\bbh$ is a Gaussian $\hbbh\sim\Normal(\mu_{\hbbh}, \Sigma_{\hbbh})$ with 
     \begin{align*}
         \mu_{\hbbh}&=(I-C)\mu_\bbf+C\mu_\bbh&\Sigma_{\hbbh}=(I-C^2)\Sigma_\bbf+C\Sigma_\bbh C\,.
     \end{align*}
 \end{prop}
\begin{proof}
    Under the GHMC we have:
    \begin{align*}
        \hbbh(q)=\cT(\bbh)(q)=& \int_\bbP\left[ (2\pi)^{-d/2}\sqrt{\Det(\Sigma_\bbh)^{-1}}\exp\left(-\frac{1}{2}(Q-\mu_\bbh)^\top\Sigma_\bbh^{-1}(Q-\mu_\bbh)\right)\right.\cdot
        \\
        &\cdot\left.
        (2\pi)^{-d/2}\sqrt{\Det(\Sigma_\bbg)^{-1}}\exp\left(-\frac{1}{2}(P-\mu_\bbg)^\top\Sigma_\bbg^{-1}(P-\mu_\bbg)\right)\,dp\right]
        \\
        =&(2\pi)^{-d/2}\sqrt{\Det(\Sigma_{\hbbh})^{-1}}\exp\left(-\frac{1}{2}(q-\mu_{\hbbh})^\top\Sigma_{\hbbh}^{-1}(q-\mu_{\hbbh})\right)\cdot
        \\
        &\cdot
        (2\pi)^{-d/2}\sqrt{\Det(K)}\int_\bbP 
        \exp\left(-\frac{1}{2}(\ttp-(x-X\ttq))^\top K(\ttp-(x-X\ttq))\right)\,d\ttp\\
        =&(2\pi)^{-d/2}\sqrt{\Det(\Sigma_{\hbbh})^{-1}}\exp\left(-\frac{1}{2}(q-\mu_{\hbbh})^\top\Sigma_{\hbbh}^{-1}(q-\mu_{\hbbh})\right)\,. \end{align*}
\end{proof}
\begin{cor}
    For any $\bbh\in \calg(d)$ the mean and variance of iterations of the GHMC operator $\cT^n(\bbh)$ with fixed target $\bbf$ converge to the mean and variance of the target. 
\end{cor}
In particular, when $d=1$, $C=\cos(\sqrt{\Sigma_\bbf^{-1}\Sigma_\bbg^{-1}}t)\in (0,1)$ is a scalar and:
\begin{align}\label{eqn:1dim formula}
\mu_{\hbbh}-\mu_\bbf&=C(\mu_{\bbh}-\mu_\bbf) & \Sigma_{\hbbh}-\Sigma_\bbf&=C(\Sigma_{\bbh}-\Sigma_\bbf)C
\,.
\end{align}

\subsection{Metric induced by the HMC for univariate Gaussians}
\label{sec:metric}

Any \emph{univariate} normal is determined by its mean $\mu$ and variance $\sigma^2=\Sigma$, i.e. a point on the $(\mu, \Sigma)$ half plane. Note that we use $\Sigma=\sigma^2$ as coordinate instead of $\sigma$, this is not the same as the half plane used in information geometry, see e.g.~\cite{amari2016information}, which is equivalent to the Poincar\'e half plane. As we demonstrated above, given $\bbf\sim\Normal(\mu_\bbf, S_\bbf)$, the evolution of points on this half plane driven by GHMC is given by~\eqref{eqn:1dim formula}:
\begin{align*}
     (\mu, \Sigma)\mapsto (\hat{\mu}, \hat{\Sigma})&=
     \left((1-C)\mu_\bbf+C\mu \,,\, (1-C^2)\Sigma_\bbf+ C^2\Sigma \right)\,,
\end{align*}
where $(0,1)\ni C=\cos(\tau)$ for some $\tau$ and in 1-dimensional case $C\Sigma C=C^2\Sigma$. 
The fact that this ratio of contraction is independent on the step and on the position allows us to view it as time 1 image of a continuous motion from $(\mu(0), \Sigma(0))$ to $(\mu_\bbf, \Sigma_\bbf)$ along the parabola:
\begin{align*}
t\mapsto(\mu(t), \Sigma(t))&=((1-C^t)\mu_0+ C^t \mu_\bbf,(1-C^{2t})\Sigma_0+ C^{2t} \Sigma_\bbf\\
\mu\mapsto \Sigma&=\Sigma_\bbf+(\Sigma_0-\Sigma_\bbf)\cdot \frac{(\mu-\mu_\bbf)^2}{(\mu(0)-\mu_\bbf)^2}\,.
\end{align*}
This further induces a distance between any two points, $(\mu_0, \Sigma_0)$ and $(\mu_1, \Sigma_1)=(\hat{\mu}, \hat{\Sigma})$ in the $(\mu, \Sigma)$ half plane, defined by the parabolic curve connecting them, 
\begin{align*}
d_{H}((\mu_0, \Sigma_0), (\mu_1, \Sigma_1))
:=&
\int_0^{|\mu_1-\mu_0|}
\sqrt{ 1+\left(2 \frac{\Sigma_1-\Sigma_0}{(\mu_1-\mu_0)^2}\right)^2 z^2
}\,dz\\
\underset{w=|\mu_1-\mu_0|z}{=}&\int_0^1\sqrt{(\mu_1-\mu_0)^2+(2w(\Sigma_1-\Sigma_0))^2}\,dw
\end{align*}
Set   $N_H(m, \Sigma)=\int_0^1\sqrt{m^2+(2w\Sigma)^2}\,dw=d_H((0,0), (m, \Sigma))$.

\begin{rem}[Distance $d_H$]\ 
\begin{itemize}
\item
The function $N_H$ is a norm and $d_H(P_0, P_1)=N_H (P_1-P_0)$, where $P_j=(\mu_j, \Sigma_j)$. 
The norm $N_H$ is an integral of the standard norm in $\Real^2$ of the vectors $(m, 2w\Sigma)$, $w\in[0,1]$, hence it is positively homogeneous, non-negative and 0 only when $\mu_0=\mu_1$ and $\Sigma_0=\Sigma_1$. The triangle inequality is inherited from the standard norm.  The function  $d_H$ is therefore a distance. 
\item
As an integral of standard distances $d_H$  is invariant under translations and symmetries with respect to vertical and horizontal axes, as such transformations do not change the values of $|\mu_1-\mu_0|$ and $|\Sigma_1-\Sigma_0|$.
\item
Additionally $N_H(m, 0)=|m|$ and $N_H(0,\Sigma)=|\Sigma|$. When the points are connected by horizontal or vertical intervals (degenerate parabolas) their distance is the same as the standard one. 
\end{itemize}
\end{rem}

\subsection*{Relation to optimal transport}

As we see that, with fixed starting point $h_0$, the outputs $h_n$ of the $n$-th GHMC step, $n\ge 1$, will always corresponds to a point on the quadratic curve in of the $(\mu, \Sigma)$ half plane. Theorem 7.2.2 in~\cite{ambrosio2006gradient} states that for any constant speed geodesic in the space of the probability measures, there is an optimal transport plan. Thus, we can conclude that the geodesic we identified, i.e. the parabolic curve, defines a optimal transport from the initial distribution of the the GHMC algorithm to the target distribution, and GHMC algorithm can be viewed as a discretization of the optimal transport.

\subsection*{An alternative distance}

For a point $P=(\mu, \Sigma)$ on the half plane parameterized univariate normal variables, let $R(P)=R(\mu, \Sigma)=\sqrt{\mu^2+|\Sigma|}
$.
Then as $\sqrt{a+b}\le \sqrt{a}+\sqrt{b}$ and by standard triangle inequality, with points $A=(\mu_a, \Sigma_a)$ and $B=(\mu_b, \Sigma_b)$, we have,
\begin{align*}
R(A+B)&=\sqrt{(\mu_a+\mu_b)^2+(\sqrt{|\Sigma_a+\Sigma_b|})^2}\le
\sqrt{(\mu_a+\mu_b)^2+(\sqrt{|\Sigma_a|}+\sqrt{|\Sigma_b|})^2}
\\
&\le \sqrt{\mu_a^2+(\sqrt{|\Sigma_a|})^2}+\sqrt{\mu_b^2+(\sqrt{|\Sigma_b|})^2}=R(A)+R(B)\,.
\end{align*}
Define $d_R(P_0, P_1):=R(P_0-P_1)$, 
then the above inequality for $R$ translates to the triangle inequality of $d_R$, and ensures that $d_R$ is a distance. Note that, due to the lack of homogeneity, \emph{$R$ is not a norm}.

A key property of distance $d_R$ is that the contraction under $\cT$ between $P$ and the fixed point $F=\cT{F}$ 
becomes an explicit decrease by $C$ in the distance $d_R$, that is,
\begin{align*}
d_R(\cT(P), \cT(F))&=C R(P-F)=C d_R(P,F)   \qquad\text{as}\\
R(\cT(P)-\cT(F))&=
R(C(\mu_P-\mu_F), C^2(\Sigma_P-\Sigma_F))
=C\sqrt{(\mu_P-\mu_F)^2+(|\Sigma_P-\Sigma_F|)}
\end{align*}

\section{GHMC with Random Potentials}
\label{sec:SHMC}

In this section, we derive convergence results on HMC with random potential energy functions at each step. In Sec.~\ref{sec:MV_Conv}, weak convergence of such HMC is established by a connection with random difference equation and stochastic fixed point equation. Then, in Sec.~\ref{sec:UV_Conv}, the limit of the weak convergence is explicitly characterized in the univariate case. 

\subsection{Convex limit set when targets may vary}
\label{sec:convex_hull}

Let $\Phi$, the set of parameters $(\mu_\bbf, \Sigma_\bbf)$ of target Gaussian distributions, represent the set of potential energy functions in the GHMC model,  which define quadratic forms in the exponents of the target. For a pair of parameters $(\mu_\bbf, \Sigma_\bbf)$ and a pair  $(\mu_\bbg, \Sigma_\bbg)$ of parameters of an auxiliary distribution, let $C, S$ denote $\cos(D)$ and $\sin(D)$ with $D=t\sqrt{\Sigma_\bbf^{-1}\Sigma_\bbg^{-1}}$, as described in Theorem~\ref{thm: GHMC on cG}. We  have $C, S>0$ and $C^2+S^2=I$. 
Denote $\cF_\mu$ and $\cF_\Sigma$ as vector and matrix convex hulls of mean vectors and covariance matrices of
potential energy functions in~$\Phi$:
\begin{align*}
\cF_\mu:=&{\left\{\sum_{j=1}^J V_j^\top \mu_j\,\bigg|\, J\in \Natural;\,  V_j\ge 0, \,j=1,\ldots, J;\, \sum_{j=1}^J V_j^\top=I, \right\}}\,,\\
\cF_\Sigma:=&{\left\{\sum_{j=1}^J V_j^\top \Sigma_jV_j\,\bigg|\, J\in \Natural;\, j=1,\ldots, J;\, \sum_{j=1}^J V_j^\top V_j=I\right\}}\,,
\end{align*}
where $(\mu_j, \Sigma_j)\in \Phi$ for every $j\in J$. Note that  $V^\top V$ is always non-negative.

\begin{prop}\label{prop:dist to convex}
Suppose that given Gaussian $\bbh$ with  parameters $(\mu_\bbh, \Sigma_\bbh)$, the distribution $\hbbh$ with parameters $(\mu_{\hbbh}, \Sigma_{\hbbh})$ its the image of $\bbh$ under $GHMC$ with target $\bbf$ with parameters $(\mu_\bbf,\Sigma_\bbf)\in\Phi$ and an arbitrary Gaussian auxiliary.
The for some $C,S>0$ with $C^2+S^2=I$,
then the following relation holds,
\begin{align*}
 \dist(\mu_{\hbbh}, \cF_\mu)&\le \|C\|\,\,
      \dist(\mu_{\bbh}, \cF_\mu)  \,,\\
    \dist(\Sigma_{\hbbh}, \cF_\Sigma)&\le \|C\|^2
      \dist(\Sigma_{\bbh}, \cF_\Sigma),
\end{align*}
where $\|C\|$ is the supremum of $\|\cos(D)\|$ over all pairs of target distributions in $\Phi$ and arbitrary auxiliary distributions, and $\dist(a, B)=\inf\{\|a-b\|:b\in B\}$.
\end{prop}
\begin{proof}
Let $\cS(\bbh)\in\overline{\cF_\Sigma}$ satisfies $\|\cS(\bbh)-\Sigma_\bbh\|=d(\Sigma_\bbh, \cF_\Sigma)$, it is well defined by convexity. Let us further assume that $\cS(\bbh)\in\cF_\Sigma$,
in case when $\cS(\bbh)\not\in\cF_\Sigma$ one proceeds with the standard approximation argument.  We can write $\cS(\bbh)=\sum_jV^\top_j\Sigma_j V_j$ for some $V_j$'s satisfying the defining conditions of $\cF_\Sigma$ and $(\mu_j, \Sigma_j)\in\Phi$. 
Let $\hat{\cS}=C \cS(\bbh) C +S \Sigma_\bbf S$, 
then $\hat{\cS}=\sum\limits_j C V_j^\top \Sigma_j V_j C+S \Sigma_\bbf S$.
Define $\hat{V}_j=V_j C$ and $\hat{V}_\cdot=S$.  By symmetry and positivity of $C$ and $S$, we have $\hat{V}^\top_j \hat{V}_j\ge 0$, $\hat{V}_\cdot^\top \hat{V}_\cdot\ge 0$,
and 
 \[
 \sum_j C V_j^\top V_j C+\hat{V}_\cdot^\top\hat{V}_\cdot=C\left( \sum_j V_j^\top V_j\right) C+SS=C^2 I +S^2=I\,.
 \]
Therefore, $\hat{\cS}\in \cF_\Sigma$. 
Furthermore, we have, 
\begin{align*}
d(\Sigma_{\hbbh}, \cF_\Sigma)\le \|\Sigma_{\hbbh}-\hat{\cS}\|=
\|C(\Sigma_\bbh-\cS(\bbh))C\|\le \|C\|^2\cdot d(\Sigma_{\bbh}, \cF_\Sigma).
\end{align*}
Similarly, let $\cM(\bbh)\in \cF_\mu$ satisfies that $\|\cM(\bbh) -\mu_\bbh\| =d(\mu_\bbh, \cF_\mu)$. 
Define $\hat{\cM}=C\cM(\bbh)+(I-C)\mu_\bbf$, it satisfies, 
\begin{align*}
\hat{\cM}=\sum_jCV_j^\top\mu_j+(I-C)\mu_\bbf =\sum_j\hat{V}_j^\top\mu_j+\hat{V}_\cdot^\top\mu_\bbf\in \cF_\mu
\end{align*}
as $C\left(\sum\limits_j V_j^\top\right)+(I-C)=I$ and all terms are positive matrices. Therefore, 
$d(\mu_{\hbbh}, \cF_\mu)=\|\mu_{\hbbh}-\cM(\hbbh)\|\le \|\mu_{\hbbh}-\hat{\cM}\|=\|C(\mu_\bbh-\cM(\bbh))\|\le \|C\|\cdot d(\mu_{\bbh}, \cF_\mu)$. 
\end{proof}

\begin{rem}
   For any (arbitrary) trajectory we have $\dist(\Sigma_{(n)},\cF_\Sigma)\le \|C\|^{2n}\dist(\Sigma_0, \cF_\Sigma)$ and
   $\dist(\mu_{(n)},\cF_\mu)\le \|C\|^{n}\dist(\mu_0, \cF_\mu)$. 
   When the supremum $\|C\|<1$  the two first moments of the  trajectory $\cT^{(n)}\bbh$ approach at least exponentially fast the convex hull of the moments of the targets. This condition is satisfied when the sets of parameters of target and auxiliary distributions is bounded. Remark that in case of the variance that means: bounded away from 0, as $C$ is defined by the inverses of the variance. Clearly it excludes the delta distributions. 
   
   The set of possible accumulation points depends on the value of $\|C\|$'s , it may be very thin when the value of the supremum is small, see, e.g.~\cite{jordan2007hausdorff}. 
\end{rem}
\commentOUT{
It is clear that $(\mu_{\hbbh}, \Sigma_{\hbbh})\in \cF(\{\mu_k\}\cup\{\mu_\bbh\})\times\cF(\{\Sigma_k\}\cup\{\Sigma_\bbh\})$. We also proved that under the iterations the distance to the convex hulls decreases geometrically (provided that $C_k$ are separated from $I$ uniformly).}

\subsection{The General Multivariate Case}
\label{sec:MV_Conv}

Recall that in~\eqref{eqn:thm:mv contraction}, the evolution of the mean vector $\mu$ and covariance matrix $\Sigma_{{\bbh}}$ satisfies,
\begin{align*}
    \mu_{\hbbh}-\mu_\bbf&=C(\mu_\bbh-\mu_\bbf)&\text{and}&&\Sigma_{\hbbh}-\Sigma_{\bbf}&=C^\top(\Sigma_{{\bbh}}-\Sigma_{\bbf})C.
\end{align*}
with the matrices $C=\cos( D)$ and $S=\sin(D)$  with $D=t\sqrt{\Sigma_\bbf^{-1}\cdot\Sigma_\bbg^{-1}}$ for some $t>0$ satisfies $0<C<I$ and $C^2+S^2=I$.

Consider GHMC with random potential energy, at each step $k$, potential energy functions are independently and randomly selected using a set of target distributions. For simplicity of exposition, we assume that the kinetic energy function is fixed. More specifically, let us denote the $M_k$ and $S^2_k$ as the random vector and covariance matrix chosen independently and at each step $k\ge 1$, with identical distribution.

\medskip
\noindent 
Therefore, equation~\eqref{eqn:thm:mv contraction} implies that the $k$-th step of GHMC produces \textbf{a random variable, denoted by\marginpar{$(*)$} $O_{k+1}$}, which is a multivariate Gaussian conditioning on the random mean vector $\mu_{k+1}$ and the covariance matrix $\Sigma_{k+1}$ satisfying:
\begin{align*}
  \mu_{k+1}-M_{k+1}&=C_{k+1}(\mu_k-M_{k+1})&\text{and}&&\Sigma_{k+1}-S^2_{k+1}&=C_{k+1}^\top(\Sigma_{k}-S^2_{k+1})C_{k+1}.
\end{align*}
where $C_k$ is $\sin(D_k)$ and $D_k=t\sqrt{(S^2_k)^{-1}\cdot\Sigma_\bbg^{-1}}$.
After reexamining  this expression we see, that 
\begin{align*}
\mu_{k+1}=&M_{k+1}+C_{k+1}(\mu_k-M_{k+1})
\\
=&C_{k+1}\mu_k + (I-C_{k+1})M_{k+1}
\\
=&C_{k+1}C_k\mu_{k-1} + C_{k+1}(I-C_{k})M_{k}+(I-C_{k+1})M_{k+1}
\\=&\sum_{\ell=0}^k \left(\prod_{j=\ell}^{k-1}C_{j+2}\right)(I-C_{\ell+1})M_{\ell+1}.
\end{align*}
Moreover, 
\begin{align*}
\Sigma_{k+1}&=S^2_{k+1}+C_{k+1}^\top(\Sigma_{k}-S^2_{k+1})C_{k+1}
=C_{k+1}^\top \Sigma_{k}C_{k+1}+
S^2_{k+1}-C_{k+1}^\top S^2_{k+1}C_{k+1}.
\end{align*}
The recursion $\mu_{k+1}=C_{k+1}\mu_k + (I-C_{k+1})M_{k+1}$ allows us to invoke the results on random difference equation started from~\cite{kesten1973random,burdzy2022stochastic} to conclude the convergence. Similarly, if $\Sigma_k$ is treated a $d^2$ dimensional vector, then $C_{k+1}^\top \Sigma_{k}C_{k+1}$ can be viewed as $d^2\times d^2$ matrix, whose $((j, \ell), (m,n))$-th entry is $C_{m j} C_{n \ell}$, left multiplies the vector $d^2$-dimensional vector $\Sigma_k=(\Sigma_k)_{m,n}$.

\subsubsection*{On Random Difference Equations}

Random difference equations defined as $X_{n+1}=A_nX_n +B_n$ for vectors $X_n$ given that $(A_n, B_n)$ are independent and identically distributed matrix and vector pair. The following result known for the one dimension case.
\begin{thms}[Theorem 2.1 in~\cite{goldie2000stability}]
Suppose that $\Prob[A_0=0]<1$ and $\Prob[A_0]=0$. Then
\begin{align*}
\sum_{n=1}^\infty |B_n|\prod_{j=1}^{n-1}|A_j|<\infty, \quad a.s.
\end{align*} 
is equivalent to 
\begin{align*} 
\prod_{j=1}^{n} A_j \to 0\text{ as }(n\to \infty), \quad \hbox{ and } \int_{(1,\infty)} \frac{\log b}{f_A(\log b)} \Prob_{B}(db) <\infty.
\end{align*}
Moreover, these conditions implies that for a given $X_0$, independent of $(A_n, B_n)$, $X_n$ converges in distribution to $S:=\sum_{n=1}^\infty B_n\prod_{j=1}^{n-1}A_j$. 
\end{thms}
Moreover, the following condition is a known (Corollary 4.1 in~\cite{goldie2000stability} sufficient condition for the convergence,
\begin{align*}
\ex[\log|A_1|]<0, \quad \hbox{ and }\ex[\log^+|B_1|]<\infty.
\end{align*}
Thus, this is rather weak conditions. 

On the multi-dimensional case, a sufficient condition for convergence is identified in~\cite{10.1214/aoms/1177705909}. Under the mild condition of $\ex[\log^+\|A_1\|]<\infty$ ($\|\cdot\|$ as the operator norm), from Kingman's subadditive ergodic theorem, there exists $\al \in[-\infty, \infty)$, such that,
\begin{align*}
\al=\lim_{n \to \infty}n^{-1} \log \left\|\prod_{j=1}^n A_j\right\|, \quad a.s.
\end{align*}
If $\al <0$ and $\ex[\log^+\|B_1\|]<\infty$ (here $\|\cdot\|$ is just the Euclidean norm), then $X_n$ converge to a random variable $X_\infty$ in distribution. Furthermore, $X_\infty$ can be written as the a.s. limit of the following random series,
\begin{align*}
\sum_{n=1}^\infty\left(\prod_{j=1}^{n-1} A_j\right)B_n.
\end{align*}
$X_\infty$ satisfies the following stochastic fixed point equation $X_\infty\stackrel{d}{=} A_1X_\infty+B_1$. 

\subsubsection*{Convergence of the SHMC}

From the above descriptions, we would like to have the following assumption. 
\begin{asm}
\label{asm:Lyapunov_exponent}
The auxiliary distribution and the random potential can be chosen such that 
$\ex[\log^+ \|C\|]<\infty$.
\end{asm}
The Main Theorem in~\cite{10.1214/aoms/1177705909} thus allows us to conclude:
\begin{thm}
\label{thm:convergence_SHMC}
Under the Assumption~\ref{asm:Lyapunov_exponent}, the sequence of random variables $O_k$, representing the outcome of the GHMC algorithm, converges in distribution. 
\end{thm}
The notation $O_k$ was introduced at $(*)$ above.
\begin{proof}
At each step $k$, as we demonstrated, $(\mu_k, \Sigma_k)$ as a $N(N+1)$ random vector, satisfies the random recursion. Under the condition of negative Lyapunov exponent on the multiplier and finite logarithm moments on the increment, $(\mu_k, \Sigma_k)$ converge in distribution to $(\mu_\infty, \Sigma_\infty)$ which indicates that, as $k\to \infty$
\begin{align*}
\exp\left[\I [\xi^\top \mu_k + \eta^\top \Sigma_k]\right]\to \exp\left[\I [\xi^\top \mu_\infty + \eta^\top \Sigma_\infty]\right].
\end{align*}
The outcome of the $k$-th step is a $N$-variate Gaussian random variable with mean $\mu_k$ and covariance matrix $\Sigma_k$, thus its characteristic function is $\phi_k(\psi)$, which is the above with $\xi=\psi$ and $\eta= \frac12\psi \psi^\top$. Therefore, we obtain the desired convergence. 
\end{proof}

\subsection{Power series with i.i.d. coefficients, convergence in the univariate case}
\label{sec:UV_Conv}

In this section, we consider the univariate case. 
Let $\calm$ denote the distribution of $(m_j, s_j^2)\in\Real\times \Real^{+}$, $j\in J$ with probabilities $p_j$, denote the expectations $m=\ex [m_j]=\sum m_j p_j\,$ and $s^2=\ex[s_j^2]=\sum s_j^2 p_j$.
Suppose that at each time $k=1,2,\ldots,$, the target normal random variables are randomly chosen from $\calm$, the (finite) set of $\{ (m_j, s_j), j \in J\}$ according to probabilities $p_j$, independently from the history. Then, denoting the random selection as (iid) random variables $(M_k, S^2_k)\sim\calm$,  
\begin{align*}
\begin{pmatrix}\mu_{k+1} \\ \sigma_{k+1}^2\end{pmatrix} = \begin{pmatrix} \alpha \mu_{k} \\ 
\alpha^2 \sigma_{k}^2\end{pmatrix} + \begin{pmatrix} (1-\alpha) M_{k+1} \\ (1-\alpha^2) S_{k+1}^2\end{pmatrix},
\end{align*}
with the parameter $t$ chosen according to $(M_k, S_k^2)$ such that $\alpha$ is always fixed.  Hence, 
\begin{align*}
\begin{pmatrix}\mu_{k} \\ \sigma_{k}^2\end{pmatrix} &= \begin{pmatrix}\alpha^{k} \mu_0  \\ \alpha^{2(k)} \sigma_0^2 \end{pmatrix} + \sum_{n=0}^{k-1} \begin{pmatrix}(1-\alpha)\alpha^n  M_{k-n}\\ (1-\alpha^2)\alpha^{2n} S_{k-n}^2\end{pmatrix}
\\
&= \begin{pmatrix}\alpha^{k} \mu_0  \\ \alpha^{2(k)} \sigma_0^2 \end{pmatrix} + \sum_{j=1}^{k} \begin{pmatrix}(1-\alpha)\alpha^{k-j}  M_{j}\\ (1-\alpha^2)\alpha^{2(k-j)} S_{j}^2\end{pmatrix}.
\end{align*}
Then 
\begin{align*}
    \ex[\mu_k]&=\alpha^k\ex[\mu_0]+(1-\alpha)\sum_{n=0}^{k-1}\alpha^n \ex[M_{k-n}]
=
    \alpha^k\ex[\mu_0]+(1-\alpha)\sum_{n=0}^{k-1}\alpha^n m    
    \\
    &=\alpha^k\ex[\mu_0]+(1-\alpha^k)m
    \\
    \ex[\sigma_k^2]&=(\alpha^2)^k\sigma_0^2+(1-(\alpha^2)^k)s^2\,.
\end{align*}
\subsection*{The characteristic function}
Define:
\begin{align}\label{eqn:Phi Lambda}
    \Phi(\psi, \zeta) &= \ex\left[\exp\left(\I(\psi M + \zeta S^2)\right)\right]
    &
    \Lambda(\xi)&=\log\left[\Phi(a \xi, \frac{\I}{2}b\xi^2)\right], 
    &\I^2=-1
    \,,
 \end{align}
 where the expectation is taken with respect to the distribution of $M$ and $S^2$. Consider the characteristic function $f_k(\psi, \zeta)$ of the distribution $(\mu_k, \sigma_k^2)$, 
\begin{align}
\label{eqn:fk}
&f_k(\psi, \zeta):=\ex\exp\left\{\I[\psi \mu_k + \zeta \sigma_{k}^2]\right\}
\\ \nonumber 
&=
\ex\exp\left\{\I\left[\psi  \left(\alpha^{k}\mu_0 + \sum_{j=1}^{k} (1-\alpha)\alpha^{k-j}  M_{j} \right) + \zeta \left(\alpha^{2k}\sigma_0 + \sum_{j=1}^{k} (1-\alpha^2)\alpha^{2(k-j)}  S^2_{j} \right)\right]\right\}
\\ \nonumber
&=
\exp[\I \left(\psi \alpha^{k}\mu_0+\zeta \alpha^{2k}\sigma_0
\right)]\prod_{j=1}^k \Phi( (1-\alpha)\alpha^{k-j}\psi, (1-\alpha^2)\alpha^{2(k-j)} \zeta)
\end{align}
due to i.i.d. assumptions on $(M_j, S^2_j)$. Therefore, we have, 
\begin{align*}
f_k(\psi, \zeta)=&
\exp[\I \left(\psi \alpha^{k}\mu_0+\zeta \alpha^{2k}\sigma_0
\right)]\prod_{j=0}^{k-1} \Phi( (1-\alpha)\alpha^{j}\psi, (1-\alpha^2)\alpha^{2j} \zeta).
\end{align*}
Observe that the right hand side
\begin{align*}
\prod_{j=0}^{k-1} \Phi( (1-\alpha)\alpha^{j}\psi, (1-\alpha^2)\alpha^{2j} \zeta)
\end{align*}
is in fact the characteristic function of  random series $Y_k:=\sum_{n=0}^{k-1} (\alpha^n M_n, \alpha^{2n} S_n^2)$. The almost sure convergence of $Y_k$ follows straightforwardly from three-series theorem, see e.g.~\cite{chow1988probability}, or directly from Taylor expansion of $\Phi(\psi, \zeta)$. Therefore, we can conclude that:
\begin{lem}
\label{lem:meanvar}
The sequence $(\mu_k, \sigma_k^2)$ converges in distribution to a random variable $(\mu_\infty, \sigma_\infty^2)$, whose characteristic function 
\begin{equation}\label{eqn:Psi}
    \Psi(\psi, \zeta)=\prod_{j=0}^\infty \Phi( (1-\alpha)\alpha^{j}\psi, (1-\alpha^2)\alpha^{2j} \zeta)\,.
\end{equation}
\end{lem}
\noindent
Furthermore, we know that, $\ex[Y_k]  = (\ex[\mu_k], \ex[\sigma_k^2])+ O(\alpha^k, \alpha^{2k})$.

\noindent
In addition, $Y_k$ are also uniformly integrable, hence, $\lim_{k\rightarrow \infty} \ex[Y_k]= \ex[Y_\infty]$ with~$Y_\infty$ being its almost sure limit whose characteristic function is $\Psi(\psi, \zeta)$. Thus, 
\begin{cor}
\label{cor:moments_convergence}
$\lim\limits_{k\rightarrow \infty} (\ex[\mu_k], \ex[\sigma_k^2])= \ex[Y_\infty]$.
\end{cor}
So, it suffices to understand, $\sum_{n=0}^\infty \alpha^n M_n$, where $M_n$ are i.i.d. random variable. Since $M_n$ is uniformly bounded, therefore, it is finite almost surely, more detailed convergence results for random series can be found in~\cite{kwapien2002random, arnold2013power}.

Let $X_k$ be the normal distribution with random mean and variance, i.e. $X_k\sim \N(\mu_k, \sigma_{k}^2)$, we have, 

\begin{thm}\label{thm:Xinfty}
There exists a random variable $X_\infty$, such that $X_k$ converge to $X_\infty$ in distribution and has the  characteristic function $\Psi(\xi, \frac{\I}{2}\xi^2)$. 
\end{thm}\begin{proof}
The characteristic function of $X_k$ can be written as
\begin{align*}
\ex\left[e^{\I\xi X_k}\right]=&\ex_{\calm}\left[\ex_{\N}\left[e^{\I\xi X_k}\bigg|\mu_k, \sigma_k^2\right]\right]\\
=&\ex_{\calm}\left[\exp\left(\I\xi\mu_k - \frac12 \sigma_k^2 \xi^2\right)\right]=f_k\left(\xi, \frac{\I}{2}\xi^2\right).
\end{align*}
From Lemma~\ref{lem:meanvar}, we know that $(\mu_k, \sigma_k^2)$ converge in distribution to a random variable $(\mu_\infty, \sigma_\infty^2)$, therefore, for each $\xi \in \Complex$, 
\begin{align}\label{eqn:PsiInfty}
f_k\left(\xi, \frac{\I} {2}\xi^2\right) \rightarrow \Psi\left(\xi, \frac{\I}{2}\xi^2\right)=\ex\left[\exp\left(\I\xi\mu_\infty - \frac12 \sigma_\infty^2 \xi^2 \right)\right]\,,
\end{align}
where $f_k$ is  given by Equation~\ref{eqn:fk}.
This implies that $X_k$ converges to $X_\infty$, which has the characteristic function $\Psi(\xi, \frac{\I}{2}\xi^2)$, in distribution. 
\end{proof}

\begin{cor}[Expectation of $X_\infty$]\label{lem:expect of X infty}
We have 
\begin{align*}
\ex[X_\infty]&=\ex[M]\quad\text{and}\quad
\ex[X_\infty^2]
  &=
    (\ex[M])^2 +\ex[S^2]+\frac{(1-\alpha)}{1+\alpha}\left(\,\ex[M^2]-(\ex[M])^2\,\right)\,.
\end{align*}
\end{cor}
\begin{proof}
Let $a_j=(1-\alpha)\alpha^j$ and $b_j=(1-\alpha^2)\alpha^{2j}$, then 
$\sum_{j=0}^\infty a_j=1=\sum_{j=0}^\infty b_j$ and $\sum_{j=0}^\infty a_j^2=\frac{(1-\alpha)^2}{1-\alpha^2}=\frac{1-\alpha}{1+\alpha}$,

By Theorem~\ref{thm:Xinfty} the characteristic function of $X_\infty$ is  
$\Psi(\psi, \zeta)$, with $\psi=\xi$ and $\zeta=\frac{\I}{2}\xi^2$. We note that $\Psi(0,0)=1$ and $\Phi(0,0)=1$. 
We can now represent $\Psi(\xi)$ as a product of $\Phi_j$, where $\Phi_j(\xi)=\Phi(a_j\xi, \frac{\I}{2}b_j\xi^2)$, and thus, with $\Lambda_j=\log \Psi_j$ we have
\begin{align*}
\Psi'(\xi)&=
\Psi(\xi)\cdot(\log'\Psi(\xi))=\Psi(\xi)\cdot\sum_{j=0}^\infty \Lambda_j'(\xi)\\
\Psi''(\xi)
&=\Psi(x)\cdot\left(\left(\sum_{j=0}^\infty \Lambda_j'(\xi)\right)^2+\sum_{j=0}^\infty\Lambda''_j(\xi)
\right)    
\end{align*}

Therefore, for $\xi=0$, 
as all  $\ex[M_j]=\ex[M]$ and all $\ex[S_j^2]=\ex[S^2]$  we have,
\begin{align*}
    \Psi(0)&=1 \qquad
\Psi'(0)=1\cdot \left(\sum_{j=0}^\infty a_j \I \ex[M_j]\right) = \I \ex[M]\sum_{j=0}^\infty a_j=\I\,\ex[M]\\
    \Psi''(0)&=\left(\I\sum_{j=0}^\infty  a_j\ex[M_j]\right)^2-
    \left(\sum_{j=0}^\infty b_j\ex[S_j^2]+a_j^2(\ex[M_j^2]-(\ex[M])^2)\right)
    \\
    &=\left(\I\, \ex[M]\sum_{j=0}^\infty a_j\right)^2-
    \left(\ex[S^2]\sum_{j=0}^\infty b_j\right)-
    (\ex[M^2]-(\ex[M])^2)\cdot \sum_{j=1}^\infty a_j^2
    \\
    &=
    -(\ex[M])^2 -\ex[S^2]-\frac{(1-\alpha)^2}{1-\alpha^2}\left(\,\ex[M^2]-(\ex[M])\,\right)^2
\end{align*}
Then $\ex[X_\infty]=\frac{1}{\I}\Phi'(0)=\ex[M]$ and 
$\ex[X_\infty^2]=-\Phi''(0)=(\ex[M])^2+\ex[S^2]+\frac{1-\alpha}{1+\alpha}\left(\ex[M^2]-(\ex[M])^2\right)$.
\end{proof}

\begin{rem}
    $\dx[X_\infty]=\ex[S^2]+\frac{1-\alpha}{1+\alpha}\dx[M]=\dx[S]+\frac{1-\alpha}{1+\alpha}\dx[M]+(\ex[S])^2$.
    That shows how the variability of $X_\infty$ depends  on the sum of variabilities of $S$ and $M$. 
\end{rem}


\end{document}